\documentclass[12pt]{amsart}
\usepackage{amsmath}
\usepackage[margin=1.6cm]{geometry}
\usepackage{parskip}
\usepackage{enumerate}

\usepackage{graphicx}
\usepackage{import}
\usepackage{pstricks}
\usepackage{amsfonts}
\usepackage{array}
\usepackage{amsthm}
\usepackage{thmtools}
\usepackage{amssymb}

\usepackage{hyperref}
\usepackage[capitalise]{cleveref}

\newtheorem{remark}{Remark}[section]
\newtheorem{theorem}{Theorem}[section]

\newtheorem{lemma}[theorem]{Lemma}
\newtheorem{proposition}[theorem]{Proposition}
\newtheorem{definition}{Definition}[section]
\newtheorem*{theorem A}{Theorem 1.1}
\newtheorem*{theorem B}{Theorem 1.2}
\newtheorem{statement}{Statement}[section]

\DeclareMathOperator{\GL}{GL}
\DeclareMathOperator{\Tr}{tr}
\DeclareMathOperator{\tr}{tr}
\DeclareMathOperator{\Ind}{Ind}
\DeclareMathOperator{\Ker}{Ker}

\DeclareMathOperator{\M}{M}

\newcommand{\iso}{\mathbin{\kern.15em\widetilde{\hphantom{\hspace{.6em}}}\kern-.98em\rightarrow\kern.05em}}
\newcommand{\longiso}{\mathbin{\kern.3em\widetilde{\hphantom{\hspace{1.1em}}}\kern-1.55em\longrightarrow\kern.1em}}

\newcommand{\Irr}{\operatorname{Irr}}

\newcommand{\F}{\ensuremath{\mathbb{F}}}

\newcommand{\mfg}{\mathfrak{g}}

\newcommand{\mfp}{\mathfrak{p}}

\newcommand{\mfo}{\mathcal{O}}

\newcommand{\mfO}{\mathcal{O}}

\title[The stable representations of $\GL_N$]{\textbf{The stable representations of $\GL_N$ over finite local principal ideal rings}}
\author{Nariel Monteiro}
\address{Department of Mathematics
University of California
Santa Cruz, CA 95064 U.S.A}
\email{namontei@ucsc.edu}

\begin{document}

\maketitle

\begin{abstract}
Let $\mathcal{O}$ be a discrete valuation ring with maximal ideal $\mathfrak{p}$ and with finite residue field $\mathbb{F}_{q}$, the field with $q$ elements where $q$ is a power of a prime $p$. For $r \ge 1$, we write  $\mathcal{O}_r$ for the reduction of $\mathcal{O}$ modulo the ideal $\mathfrak{p}^r$. An irreducible ordinary representation of the finite group $\mathrm{GL}_{N}(\mathcal{O}_{r})$ is called stable if its restriction to the principal congruence kernel $K^l=1+\mathfrak{p}^{l}\mathrm{M}_{N}(\mathcal{O}_r)$, where $l=\lceil \frac{r}{2} \rceil$, consists of irreducible representations whose stabilizers modulo $K^{l'}$, where $l'=r-l$, are centralizers of certain matrices in $\mathfrak{g}_{l'}=\mathrm{M}_{N}(\mathcal{{O}}_{l'})$, called stable matrices.

The study of stable representations is motivated by constructions of strongly semisimple representations, introduced by Hill, which is a special case of stable representations. In this paper, we explore the construction of stable irreducible representations of the finite group $\mathrm{GL}_{N}(\mathcal{O}_{r})$ for $N \ge 2$.
\end{abstract}

 \section{Introduction}
 
Let $\mathrm{F}$ be a non-Archimedean local field -- i.e., $\mathrm{F}$ is the field of fractions of a complete discrete valuation ring $\mathcal{O}$ with finite residue field. Let $\mathfrak{p}$ be the maximal ideal of $\mathcal{O}$ and let $\pi$ be a fixed generator of $\mfp$. Assume that the residue field $\mathcal{O/ \mfp} \simeq \F_{q}$ has characteristic $p$. We write  $\mathcal{O}_r$ for the reduction of $\mathcal{O}$ modulo $\mathfrak{p}^r$. Two good examples to keep in mind are $\mathbb{F}_{p}[t]/t^{r}$ and $\mathbb{Z}_p/p^r \simeq \mathbb{Z}/p^r$. This paper is concerned with the construction of a family of continuous complex-valued irreducible representations of the groups $\GL_{N} (\mathcal{O})$. As a maximal compact subgroup of $\GL_{N} (\mathrm{F})$, the group $\GL_{N} (\mathcal{O})$ and its representations play an important role in the representation theory of $\GL_{N} (\mathrm{F})$; see, for instance, \cite{bushnell1993admissible}. A representation of $\GL_{N} (\mathcal{O})$ is continuous if and only if it factors through a finite quotient $\GL_{N} (\mathcal{O}_r)$. We thus focus on the representations of the finite groups $\GL_{N} (\mathcal{O}_r)$.

The complex representation theory of general linear groups over the rings $\mathcal{O}_r$ has been heavily studied; see \cite{stasinski2017representations} for an overview and history of the subject. For example, two independent constructions of regular representations have been established. One is by Krakovski, Onn, and Singla \cite{krakovski2018regular}, which works whenever $p$ is not $2$, and the other is by Stasinski and Stevens \cite{stasinski2017regular}. The latter works for any $\mathcal{O}$ independent of $p$, and thus there is a complete construction of all the regular representations of $\GL_N(\mathcal{O}_r)$. On the other hand, Hill worked on constructing strongly semisimple representations \cite{hill1995semisimple}. In this paper, we introduce a new class of representations of $G_r=\GL_N(\mathcal{O}_r)$, the stable representations, and we explore how to construct them.

The previous constructions of irreducible representations of $G_r$ explore the fact that $G_r$ has an abelian normal subgroup, denoted as $K^l$, obtained as the kernel of the surjective map $G_r \to G_l$ where $l=\lceil \frac{r}{2} \rceil$. Using Clifford's theory, one can relate representations of the normal subgroup $K^l$ to representations of $G_r$. Moreover, one can show that the set of irreducible representations of $K^l$ is in bijection with $\mathfrak{g}_{l'}=\M_{N}(\mathcal{{O}}_{l'})$. Thus, we will denote an irreducible representation of $K^l$ as $\psi_M$, where $M$ is a matrix with entries in $\mathcal{O}_{l'}$. We call $M$ a \emph{stable} matrix if by possibly passing to an unramified extension $\mathcal{\hat{O}}_{l'}$ of $\mathcal{O}_{l'}$, we have that $M$ is conjugate to a matrix $A+\pi B$ with entries in $\mathcal{\hat{O}}_{l'}$ such that $A$ is in Jordan canonical form with $B$ in $Z(C_{\M_{N}(\mathcal{\hat{O}}_{l'})}(A))$. We note here the definition of a matrix in Jordan canonical form over a ring. If $R$ is a ring, then $J_n(\lambda)$ denotes an $n \times n$ matrix with all diagonal entries equal to $\lambda$, all entries just above the diagonal equal to $1$, and all other entries are $0$. We say that a matrix $A$ over a ring $R$ is in Jordan canonical form if it is block-diagonal with each diagonal block of the form $J_n(\lambda)$, for some $\lambda \in R$ and some positive integer $n$.

We say that a character $\chi$ of $G_r$ is a \emph{stable} character if $\chi$ is above $\psi_M$ for $M$ a stable matrix, i.e., $\psi_M$ is an irreducible constituent of the restriction of $\chi$ to $K^l$. Our first
main result is the construction of the stable representations of $G_r$ for $r$ even by using Clifford's theory. To be more specific, we prove the following:

\begin{theorem}
\label{theorem:Ieven}

Let $r=2l$ and $\chi$ be a stable irreducible character of $G_r=\GL_{N}(\mathcal{O}_{r})$ above $\psi_M$, where $\psi_M$ is an irreducible character of $K^l$. Then:  

\begin{enumerate}
    \item $\psi_M$ extends to a character $\tilde{\psi}_{M}$ of $G_r(\psi_M)$, the stabilizer of $\psi_M$ in $G_r$.
    \item There is an irreducible character $\rho$ of $G_r(\psi_M)/K^l$ such that $\chi=\operatorname{Ind}_{G_r(\psi_M)}^{G_r}(\tilde{\psi}_{M}\rho)$.
\end{enumerate} 

\end{theorem}

As we will see in subsequent sections, the representation theory of $G_r$ is much harder when $r$ is odd compared to when $r$ is even. When $r$ and $p$ are odd, we give a construction of some of the stable representations of $G_r$, which we call super stable representations of $G_r$. For each $\psi_M$, there is a subgroup $J_M$, with $K^l \subseteq J_M \subseteq K^{l'}$, such that any irreducible representations of $K^{l'}$ above $\psi_M$ is equal to $\sigma=\Ind_{J_{M}}^{K^{l'}}(\psi_{\hat{M}})$ with $\psi_{\hat{M}}$ an extension of $\psi_{M}$, where $\hat{M} \in \mathfrak{g}_{l}$ such that $M=\hat{M} \mod{\pi^{l'}}$.  We define an irreducible character $\chi$ of $G_r$ to be \emph{super stable} if $\chi $ is above $\sigma = \Ind_{J_{M}}^{K^{l'}}(\psi_{\hat{M}})$ with $\hat{M}$ a stable matrix of $\mathfrak{g}_{l}$. In \cref{Section:6}, we give the precise definition of super stable irreducible character of $G_r$ for when $r$ is odd, and we prove:

\begin{theorem}
\label{theorem:Iodd}
Assume that $\mathcal{O}$ has residue field of characteristic $p>2$ and $r=2l-1$. Let $\chi$ be a super stable irreducible character of $G_r$ above an irreducible character $\sigma$ of $K^{l'}$. Then:
\begin{enumerate}
    \item $\sigma$ extends to a character $\tilde{\sigma}$ of $G_r(\sigma)$, the stabilizer of $\sigma$ in $G_r$.
    \item There is an irreducible character $\rho$ of $G_r(\sigma)/K^{l'}$ such that $\chi=\operatorname{Ind}_{G_r(\sigma)}^{G_r}(\tilde{\sigma}\rho)$.
\end{enumerate} 
\end{theorem}

We briefly define the strongly semisimple representations of $G_r$ here as they are related to the stable representations. Recall we denote an irreducible representation of $K^l$ as $\psi_M$, where $M$ is a matrix with entries in $\mathcal{O}_{l'}$. Semisimple representations of $G_r$ correspond to representations above $\psi_M$ such that the image of $M \mod \mathfrak{p}$ is a semisimple matrix. In \cite[Proposition 2.3]{hill1995semisimple}, Hill gives an additive Jordan decomposition in the ring $\mathfrak{g}_{l'}$. Specifically, given a matrix $M \in \mathfrak{g}_{l'}$, it can be written as $M=s+n$ with commuting matrix $s$ semisimple and $n$ nilpotent. A matrix $s \in \mathfrak{g}_{l'}$ is called semisimple if it is conjugate to a diagonal matrix with diagonal entries in the multiplicative system of representatives in $\mathcal{O}_{l'}$ or zero.  Strongly semisimple representations of $G_r$ correspond to semisimple representations above $\psi_M$ for $M=s+n \in \mathfrak{g}_{l'}$ such that $s$ is a semisimple  matrix and $n$ a nilpotent element of $Z(C_{\M_{N}(\mathcal{{O}}_{l'})}(s))$. Note that if such $M$ is in Jordan canonical form, possibly over an unramified ring extension, then $s$ is a diagonal semisimple matrix and $n$ is a diagonal nilpotent matrix, so $n \in \pi \mathfrak{g}_{l'}$ \cite[In Proof of Proposition 3.3]{hill1995semisimple}. Thus, Hill's construction of strongly semisimple representations falls under the family of stable representations where the stable matrix $M$ is conjugate to a matrix $A+\pi B$ with the additional assumption that $A$ is a semisimple matrix. Thus, when $r$ is even, we generalize Hill's construction of strongly semisimple representations of $G_r$ \cite[Proposition 3.3]{hill1995semisimple}.

For the reader's benefit, we point out that a regular representation corresponds to a representation of $G_r$ above a $\psi_M$ for $M$ a regular matrix. Note that the matrix $\begin{bmatrix}
  1 & 1\\ 
  0 & 1+\pi
\end{bmatrix}$  is a regular matrix which is not stable. We also note that for $r=2$, all irreducible representations of $G_2$ are stable. 

\subsection*{Organization of the paper}
\Cref{Section:2} contains notation and preliminary remarks. For example, we state Clifford's theorem, which we use later to prove \cref{theorem:Ieven} and \cref{theorem:Iodd}. In \Cref{Section:3}, we review the similarity classes of $\M_N(\F_q)$. We also describe the center of the centralizer of nilpotent Jordan Blocks over a general commutative ring. In \Cref{Section:4}, we build upon the description of the centralizer of stable matrices to determine the stabilizer of the stable representations of $K^l$. In \Cref{Section:5}, we prove \cref{theorem:Ieven}. Thus, giving a complete construction of stable representations of $G_r$ when $r$ is even. In \Cref{Section:6}, we prove \cref{theorem:Iodd}. Thus, giving a complete construction of super stable representations of $G_r$ when $r$ is odd and $p$ is not $2$. In \Cref{Section:7}, we talk about a problem with the construction of strongly semisimple representations when $r$ is odd and hence the necessity of introducing super stable representations.

\subsection*{Acknowledgements}

I am extremely grateful to my advisor George McNinch for his invaluable guidance and support throughout this project. I also thank Alexander Stasinski for reading a draft of this article and providing valuable and constructive feedback. Lastly, I would like to thank the referee for their suggestions that improved both the content and the exposition in this paper. This research was partially supported by Thesis Writing Fellowship, Noah Snyder's NSF CAREER Grant DMS-1454767.

\section{Background and Notation}
\label{Section:2}
We will fix the following notations: For a group $G$ with $g,h \in G$ we denote the group commutator $ghg^{-1}h^{-1}$ by $[g,h]$ and the conjugation $ghg^{-1}$ by $^gh$.

\subsection{Rings and Fields}
Let $\mathrm{F}$ be a non-Archimedean local field with ring of integers $\mathcal{O}$. Let $\mathfrak{p}$ be the maximal ideal of $\mathcal{O}$ and $\pi$ be a fixed generator of $\mfp$. Assume that the residue field $\mathcal{O}/ \mfp$ is finite, say $k=\F_{q}$, of characteristic $p$. For $r \geq 1$, we will denote $\mathcal{O}_r$ to be $\mathcal{O}/ \mfp^r$  and $\eta_r$ for the map $\eta_r:\mathcal{O} \to \mathcal{O}_r$. Given, $\eta_1:\mathcal{O} \to k$, there is a section $\mu: k \to \mathcal{O}$ satisfying $\mu(ab)=\mu(a)\cdot\mu(b)$. We will denote $\mu_r: k \to \mathcal{O}_r$ by composing $\mu$ with $\eta_r$. Any element $a$ of $\mathcal{O}$ has a unique
representation as a series $\sum_{i=0}^{\infty} \mu(a_i)\pi^i$, where $a_i$ is an element of $k$. Similarly any 
element $a$ of $\mathcal{O}_r$ has a unique
representation as a finite series $\sum_{i=0}^{r-1} \mu_r(a_i)\pi^i$, where $a_i$ is an element of $k$. For $s \leq r$, we will fix the following section:

\begin{align*}
  \mu_{s,r}:\mfO_s & \longrightarrow \mfO_r\\
 \mu_{s,r}(\sum_{i=0}^{s-1} \mu_s(a_i)\pi^i) & =\sum_{i=0}^{s-1} \mu_r(a_i)\pi^i.
\end{align*}

\subsection{Groups}
We will be concerned with general linear groups over these various fields and rings. We will denote them as follow: Fix an integer $N\geq2$ and, for any $r\geq1$, let
\begin{align*}
G_{r} & =\GL_{N}(\mathcal{O}_{r}),\\
\mfg_{r} & =\M_{N}(\mathcal{O}_{r}),
\end{align*}
where we use $\M_{N}(R)$ to denote the algebra of $N\times N$ matrices over $R$ a commutative ring. 

From here on, we consider a fixed $r\geq2$. For any integer $i$ such that $r\geq i\geq1$, let $\rho_{i}:G_{r}\rightarrow G_{i}$ be the surjective homomorphism induced by the canonical map $\mathcal{O}_{r}\rightarrow\mathcal{O}_{i}$,
and write $K^{i}=\Ker\rho_{i}$. By computation, one can check that $K^{i}=I+\mfp^{i}\mfg_{r}$ with $I$ denoting the identity matrix. We also write $\rho_{i}$
for the corresponding homomorphism $\mfg_{r}\rightarrow\mfg_{i}$. We thus have a descending chain of  normal subgroups:
\[
G_{r}\supset K^{1}\supset\dots\supset K^{r}=\{I\}.
\]

\subsection{Clifford Theory}
We will use Clifford's theory to relate representations of the normal subgroups $K^i$ to representations of $G_r$. We will denote $\Irr(G)$ for the set of irreducible representations of a finite group $G$. In general, we have if $G$ is a finite group, with $H\subseteq G$ a subgroup and $\rho\in\Irr(H)$, we will write $\Irr(G\mid\rho)$ for the set of $\pi\in\Irr(G)$ such
that $\pi$ contains $\rho$ on restriction to $H$, i.e.,
\[
\Irr(G\mid\rho)=\{\pi\in\Irr(G)\mid\langle\pi|_{H},\rho\rangle\neq0\}.
\]
Moreover, given $N$ a normal subgroup of $G$, then $G$ acts on $N$ and thus on $\Irr(N)$ by $\rho\mapsto{}^{g}\rho$, where $^{g}\rho(n):=\rho(gng^{-1})$, for $g\in G$, $n\in N$. In this case, we define the \emph{stabilizer}
of $\rho\in\Irr(N)$ to be $G(\rho)=\{g\in G\mid{}^{g}\rho\cong\rho\}$. We will later make use of the following results from Clifford's theory of finite groups:

\begin{theorem}

\label{thm:Clifford}Let $G$ be a finite group, and $N$ a normal subgroup. Then the following hold:

\begin{enumerate}
\item \label{Clifford1} If $\pi\in\Irr(G)$, then $\pi|_{N}=e(\bigoplus_{\rho\in\Omega}\rho)$,
where $\Omega\subseteq\Irr(N)$ is an orbit under the action of $G$
on $\Irr(N)$ by conjugation, and $e$ is a positive integer.
\item \label{Clifford2}Suppose that $\rho\in\Irr(N)$. Then $\theta\mapsto\Ind_{G(\rho)}^{G}\theta$
is a bijection from $\Irr(G(\rho)\mid\rho)$ to $\Irr(G\mid\rho)$.
\item \label{Clifford3}Let $H$ be a subgroup of $G$ containing $N$,
and suppose that $\rho\in\Irr(N)$ has an extension $\tilde{\rho}$
to $H$ (i.e., $\tilde{\rho}|_{N}=\rho$). Then 
\[
\Ind_{N}^{H}\rho=\bigoplus_{\chi\in\Irr(H/N)}\tilde{\rho}\chi,
\]
where each $\tilde{\rho}\chi$ is irreducible, and where we have identified
$\Irr(H/N)$ with $\{\chi\in\Irr(H)\mid\chi(N)=1\}$.
\end{enumerate}
\end{theorem}

For proofs of the above, see \cite{isaacs2006character}, 6.2, 6.11, and 6.17, respectively.


\subsection{Structure and Characters of $K^{l}$}

The following lemma gives us information about the structure of the normal subgroups $K^i$. 

\begin{lemma}
\label{lemma:Kernel}
\cite[Proposition 2.1]{HILL1995610} Given $G_r$, we have:
\begin{enumerate}
    \item If $i \geq r/2$ then $K^i$ is abelian and the map $I+\pi^{i}x \mapsto \rho_{r-i}(x)$ is an isomorphism from $K^i$ to the additive group $\mfg_{r-i}$.
    \item For any $i \leq r-1$ the map $I+\pi^{i}x \mapsto \bar{x}$ induces an isomorphism from $K^i/K^{i+1}$ to $\mfg_{1}$,  where the bars denote reductions mod $\mfp$.
\end{enumerate}
\end{lemma}

Note by the previous lemma when $i\geq r/2$, then $K^{i}$ is abelian, and if we let $l=\lceil \frac{r}{2} \rceil$ is the maximal abelian group among the kernels $K^{i}$. From now on, we let $i\geq r/2$, that is, $i\geq l$. We will thus describe $\Irr(K^i)$.

Given $\mathrm{F}$ be the fraction field of $\mfo$, we fix an additive character $\psi:\mathrm{F}\rightarrow\mathbb{C}^{\times}$ which is trivial on $\mfo$ but not on $\mfp^{-1}$. For each $r\geq1$, we can view $\psi$ as a character of the additive group $\mathrm{F}/\mfp^{r}$ whose kernel contains $\mfo_{r}$. We will use $\psi$ and the trace form $(x,y)\mapsto\tr(xy)$ on $\mfg_{r}$ to define a character of $K^i$. For $M \in\M_{N}(\mathfrak{o}_{r})$, define a homomorphism $\psi_{M}:K^{i}\rightarrow\mathbb{C}^{\times}$ by

\begin{equation*}
\label{eq:character-psi-b}
\psi_{M}(I+\pi^{i}x)=\psi(\pi^{-r}\tr(M\pi^{i}x))=\psi(\pi^{-r+i}\tr(M x)),
\end{equation*}

for $x\in \mfg_{r}$. Since $\psi$ is trivial
on $\mfo_{r}$, $\psi_{M}$ only depends on $x$ and $M \mod\mfp^{r-i}$. Thanks to the non-degeneracy of the trace form, the map $M \mapsto\psi_{M}$ is a homomorphism whose kernel is $\mfp^{r-i}\mfg_{r}$. Thus it induces an isomorphism
\[
\mfg_{r}/\mfp^{r-i}\mfg_{r}\longiso\Irr(K^{i}),
\]
where we will usually identify $\mfg_{r}/\mfp^{r-i}\mfg_{r}$ with
$\mfg_{r-i}$. For $g\in G_{r}$ we have 
\[
\psi_{gMg^{-1}}(x)=\psi(\pi^{i-r}\tr(gM g^{-1}x))=\psi(\pi^{i-r}\tr(M g^{-1}xg))=\psi_{M}(g^{-1}xg).
\]
Thus the bijection above transforms the action of $G_{r}$ on $\mfg_{r}$ into (the inverse) conjugation of characters. We  summarize the discussion above in the following lemma:

\begin{lemma}
\label{lemma:Hillstabilizer}
\cite[Proposition 2.3]{HILL1995610}
Let $r/2 \leq i \leq r-1 $:
\begin{enumerate}
    \item The map $M \to \psi_M$ is a $G_r$-equivariant bijection from $M_n(\mathcal{O}_{r-i})$ to $Irr(K^i)$
    \item The stabilizer of $\psi_M$ in $G_r$ is $\rho_{r-i}^{-1}(C_{G_{r-i}}(M))$
    \end{enumerate}
\end{lemma}

Thus, by \cref{lemma:Hillstabilizer}, to apply Clifford's theorem, it is enough to pick characters $\psi_M$ of $K^l$, where $M$ runs through a set of representatives of the
conjugacy classes in $\mfg_{l'}$. 


\section{The center of the centralizer of nilpotent Jordan blocks}
\label{Section:3}

In this section, we recall some known results about the centralizer of nilpotent Jordan blocks. Then, we use those results to describe the center of the centralizer of those matrices. We also discuss Jordan's canonical form for split matrices, i.e., matrices whose characteristic polynomials split over $\F_q$.




If $A_i$’s for $1 \leq i \leq l$ are matrices, then we denote by $A=\oplus_{i=1}^l A_{i}$ the
block diagonal matrix with entries in a commutative ring of the form:

\begin{equation*}
\begin{bmatrix}
A_1 & 0 &  \cdots & 0 \\
0 & A_2 &  \cdots & 0 \\
\vdots  & \vdots &  \ddots & \vdots  \\
 0   & 0 & \cdots &  A_l
\end{bmatrix}.
\end{equation*} 

We call this a direct sum of matrices. Similarly, if $\mathcal{A}_1,\mathcal{A}_2, \cdots ,\mathcal{A}_l$ are sets of matrices then $\mathcal{A}=\oplus_{i=1}^l \mathcal{A}_{i}$ denotes the set of matrices  $\{ A=\oplus_{i=1}^l A_{i}| A_i \in \mathcal{A}_i \}$.




We will start by describing the primary decomposition of matrices under the action of conjugation.  For this, we will consider Jordan blocks $J_n(a)$ over commutative rings with identity. Recall that for matrices over $\mathbb{F}_{q}$, every split matrix is similar to a direct sum of Jordan blocks $\oplus_{i=1} J_{\lambda(a_i)}(a_i)$ with $a_i\neq a_j$ for all $i \neq j$. Note each Jordan block $J_{\lambda(a_i)}(a_i)$ can be written as $a_iI+N$ with $N$ a nilpotent matrix. Such a matrix $a_iI+N$ is called a split primary matrix in Jordan canonical form. Thus, the centralizer of $J_{\lambda(a_i)}(a_i)= a_iI+N$ only depends on $N$. Moreover, we can write $N=\oplus_{i=1}^l N_{n_i}$, with $l \geq 1$, as a direct sum of principal nilpotent matrices, which we define:

\begin{definition}
(Principal Nilpotent Matrix) A square matrix of order $n$ is called principal nilpotent if it is of the form:

\begin{equation*}
N_n:=\begin{bmatrix}
0 & 1 & 0 & 0 & \cdots & 0 \\
0 & 0 & 1 & 0 & \cdots & 0 \\
0 & 0 & 0 & 1 &  &  \\
\vdots  & \vdots &  & \ddots & \ddots & \vdots  \\
  & &  &  & 0 & 1 \\
0   & 0 & \cdots &  & 0 & 0 
\end{bmatrix}
\end{equation*} 

\end{definition}

The centralizer of a nilpotent matrix $N=\oplus_{i=1}^l N_{n_i}$ depends on the following family of matrices. 

\begin{definition}
(Upper Toeplitz Matrix) A square matrix of order $n$ over a ring $R$ is
called \emph{upper Toeplitz}, denoted as $T_{n \times n}$ if it is of the form

\begin{equation*}
T_{n \times n} = \begin{bmatrix}
a_1 & a_2 & \cdots & a_{n-1} & a_{n}  \\
0 & a_1 & a_2 &  & a_{n-1} \\
\vdots  & \ddots & \ddots & \ddots & \vdots  \\
  & \cdots &  & a_1 & a_2 \\
0   & \cdots & 0 & 0 & a_1 
\end{bmatrix}
\end{equation*} 

\end{definition}

Any such upper Toeplitz matrix will be denoted as $T_{n \times n}=(a_1,a_2, \dots, a_n)$. For $l \leq n$, given $T_{n \times n}=(a_1,a_2, \dots, a_n)$, we define $T_{(n \times n)_l}:=(a_1, \dots, a_l)$, i.e., the upper Toeplitz matrix of order $l$ with the same entries as $T_{n \times n}$ up to the $l^\text{th}$ coefficient.

\begin{definition}
(Rectangular Upper Toeplitz Matrix) A matrix of order
$n \times m$ over a ring $R$ is called a \emph{rectangular upper Toeplitz matrix}, denoted as $T_{n \times m}$, if it is of the form

$$\begin{bmatrix} 
0_{n \times m-n}   & T_{n \times n}
\end{bmatrix} \text{if $n \leq m$ or } 
\begin{bmatrix} 
 T_{m\times m} \\
 0_{n-m \times m}   &
\end{bmatrix} \text{if $n \geq m$}$$

where $T_{s \times s}$, for a natural number $s$, is the upper Toeplitz matrix of order $s$.
\end{definition}

Considering this type of matrices, we have: 

\begin{lemma}
\label{lemma:singla}
\cite[Lemma 4.9]{SinglaPooja2010Orog} Let $n_1, n_2, \cdots , n_l$ be a sequence of natural numbers, such that
$n = n_1 + n_2 + \cdots + n_l$. Let $R$ be a commutative ring with identity, and let
$A=\oplus_{i=1}^l N_{n_i}$. Then the centralizer, $C_{\M_n(R)}(A)$, of $A$ in $\M_n(R)$ consists of all matrices of the form:

\begin{equation*}
\begin{bmatrix}
T_{n_1 \times n_1} & T_{n_1 \times n_2} & \cdots & T_{n_1 \times n_l} \\
T_{n_2 \times n_1} & T_{n_2 \times n_2} & \cdots & T_{n_2 \times n_l} \\
\vdots  & \vdots  & \ddots & \vdots  \\
T_{n_l \times n_1} & T_{n_l \times n_2} & \cdots & T_{n_l \times n_l} 
\end{bmatrix}
\end{equation*} 
which are called \emph{block upper-Toeplitz} matrices.
\end{lemma}

To classify the center of the centralizer of $A$, we need the following lemma.

\begin{lemma}
\label{lemma:matrixcomp}
Given a commutative ring $R$ with identity:

\begin{enumerate}
    \item  Let $V_{n_j\times n_i}$ and $V_{n_i\times n_j}$ be rectangular upper Toeplitz matrices. If $T_{n_i\times n_j} V_{n_j\times n_i}= V_{n_i\times n_j} T_{n_j\times n_i}$ for all pair of matrices $T_{n_i\times n_j}$ and  $T_{n_j\times n_i}$ then $V_{n_j\times n_i} =0$ and $V_{n_i\times n_j} =0$. 
    \item Let $V_{n_j\times n_j}$ and $S_{n_j\times n_j}$ be upper Toeplitz matrices. If $T_{n_j\times n_j} V_{n_j\times n_j}= S_{n_j\times n_j} T_{n_j\times n_j}$ for all matrices $T_{n_j\times n_j}$ then $V_{n_j\times n_j} =S_{n_j\times n_j}$.
\end{enumerate}
\end{lemma}

\begin{proof}
The proof follows by matrix computation. 
\end{proof}

Using the previous two lemmas, we obtain the following description of the center of the centralizer of $A=\oplus_{i=1}^l N_{n_i}$:

\begin{theorem}
\label{theorem:center}
Let $A=\oplus N_{n_i}$ where $N_{n_i}$ are principal nilpotent matrices over a commutative ring $R$ with identity. Then any matrix in the center of the centralizer of $A$, denoted as $Z(C_{\M_n(R)}(A))$, has the form: \begin{equation*}
V = 
\begin{bmatrix}
V_{n_1 \times n_1} & 0 & \cdots & 0 \\
0 & V_{n_2 \times n_2} & \cdots & 0 \\
\vdots  & \vdots  & \ddots & \vdots  \\
0 & 0 & \cdots & V_{n_l \times n_l} 
\end{bmatrix}
\end{equation*} 

where each $V_{n_j \times n_j}$ is an Upper Toeplitz Matrix such that given $n_s\geq n_j$ then $V_{n_j \times n_j}=V_{(n_s \times n_s)_{n_j}}$. 
\end{theorem}

\begin{proof}
Let $V \in Z(C_{\M_n(R)}(A))$. Since $V$ is in the centralizer of $A$, it is a block upper-Toeplitz matrix, by \cref{lemma:singla}, of the form:

\begin{equation*}
\begin{bmatrix}
V_{n_1 \times n_1} & V_{n_1 \times n_2} & \cdots & V_{n_1 \times n_l} \\
V_{n_2 \times n_1} & V_{n_2 \times n_2} & \cdots & V_{n_2 \times n_l} \\
\vdots  & \vdots  & \ddots & \vdots  \\
V_{n_l \times n_1} & V_{n_l \times n_2} & \cdots & V_{n_l \times n_l} 
\end{bmatrix}
\end{equation*} 

First, we will show that $V_{n_i \times n_j}=0$ whenever $i \neq j$. Note for any $T \in C_{\M_n(R)}(A)$ we have $C=VT=TV$. Thus we have: 

\begin{equation}
\label{6equation:matrix}
    C_{n_i \times n_i}=\sum_{k=1}  T_{n_i \times n_k} V_{n_k \times n_i}=\sum_{k=1}  V_{n_i \times  n_k} T_{n_k \times n_i}  
\end{equation}

Note \cref{6equation:matrix} works for any $T$, so we can make the assumption that $T_{n_i \times n_k}=0$ and $T_{n_k \times n_i}=0$ for $k \neq j$. So \cref{6equation:matrix} becomes  $C_{n_i \times n_i}= T_{n_i\times n_j} V_{n_j\times n_i}= V_{n_i\times n_j} T_{n_j\times n_i}$. Since this works for all pair of matrices $T_{n_i\times n_j}$ and  $T_{n_j\times n_i}$, we have that $V_{n_i\times n_j}=0$ by \cref{lemma:matrixcomp}.

Thus $V$ is a block diagonal matrix. We are left to show that given $n_s\geq n_j$ then $V_{n_j \times n_j}=V_{(n_s \times n_s)_{n_j}}$. Since $V_{n_i \times n_j}=0$ whenever $i \neq j$, we have $C_{n_s\times n_j}= T_{n_s\times n_j} V_{n_j\times n_j}= V_{n_s\times n_s} T_{n_s\times n_j}$. Thus, we have:

\begin{equation*}
        C_{n_s\times n_j}=\begin{bmatrix}
           C'_{n_j\times n_j}\\           
           0
          \end{bmatrix}=\begin{bmatrix}
           T_{n_j\times n_j}\\           
           0
          \end{bmatrix} \cdot V_{n_j\times n_j}= V_{n_s\times n_s} \cdot 
          \begin{bmatrix}
           T_{n_j\times n_j}\\           
           0
         \end{bmatrix}
     \end{equation*}

Note by matrix computation $C'_{n_j\times n_j}=T_{n_j\times n_j}V_{n_j\times n_j}=V_{(n_s \times n_s)_{n_j}} T_{n_j\times n_j}$.  Since this works for all matrices $T_{n_j\times n_j}$, we have that $V_{n_j\times n_j}=V_{(n_s \times n_s)_{n_j}}$ by \cref{lemma:matrixcomp}. Thus, $V$ has the proposed form.
\end{proof}

\section{Stabilizer of a stable irreducible character of $K^l$}

\label{Section:4}

In this section, we will be concerned with finding the stabilizer of irreducible characters of $K^l$ coming from stable matrices. Recall that any such irreducible character is defined as $\psi_{M}:K^{l}\rightarrow\mathbb{C}^{\times}$ by $\psi_{M}(I + \pi^{l}x)=\psi(\pi^{l'}\tr(M x))$ for a specific matrix $M$ in $\mathfrak{g}_{l'}$ and the stabilizer of $\psi_M$ in $G_r$ is $\rho_{l'}^{-1}(C_{G_{l'}}(M))$ by \cref{lemma:Hillstabilizer}. We will show that when $M$ is a stable matrix, $G_r(\psi_M)=K^{l'}C_{G_r}(\mu_{l',r}(M))$ where $\mu_{l',r}(M)$ is a matrix with entries in $\mfO_r$ obtained from the section $\mu_{l',r}: \mfO_{l'} \to \mfO_r$. Thus, $M = \mu_{l',r}(M) \mod{\mathfrak{p}^{l'}}$.

Recall that any element $a$ of $\mfO_{l'}$ has a unique
representation as a finite series $\sum_{i=0}^{l'-1} \mu_{l'}(a_i)\pi^i$, where $a_i$ is an element of $k$. Similarly, any matrix $M$ with entries in $\mfO_{l'}$ has a unique representation as a finite series. Thus, we will write any matrix $M$ with entries in $\mfO_{l'}$ as $A+\pi B$ with $A=\mu_{l'}(\bar{M})$ where the bars denote reductions mod $\mfp$. If $\bar{M}$ is in Jordan canonical form, we have that $A=\mu_{l'}(\bar{M})$ is in Jordan canonical form. Thus, note that by possibly passing to an unramified extension, any matrix of $\mathfrak{g}_{l'}$ is conjugate to a matrix $A+\pi B$ such that $A$ is in Jordan canonical form. 

\begin{definition}
Let $M$ be  a matrix in $\mathfrak{g}_{s}$, for arbitrary integer $s\geq 1$. We call $M$ a stable matrix if by possibly passing to an unramified extension $\mathcal{\hat{O}}_{s}$ of $\mathcal{O}_{s}$ we have that $M$ is conjugate to a matrix $A+\pi B$ with entries in $\mathcal{\hat{O}}_{s}$ such that $A$ is a matrix in Jordan canonical form and $B \in Z(C_{\mathfrak{g}_{s}}(A))$.
\end{definition}

For this set of matrices over $\mathfrak{g}_{l'}$,, we will study the stabilizer of $\psi_M$ in $G_r$  which is $\rho_{l'}^{-1}(C_{G_{l'}}(M))$. We recall a general fact about centralizers of matrices. 

\begin{lemma}
\label{lemma:blockcentral} \cite[Lemma 4.5]{SinglaPooja2010Orog}
Let $R$ be a commutative ring with unity. Let $a_1, a_2, \cdots, a_l$
be elements of $R$ such that for all $i \neq j$, $a_i-a_j$ is invertible in $R$. Let $A=\oplus_{i=1}^l A_{i}$
be a square matrix of size $n$, where $A_i$’s are upper triangular
matrices of size $n_i$. Assume that all diagonal entries of $A_i$ are equal to $a_i$. Then, $C_{GL_n(R)}(A) = \oplus_{i=1}^l
C_{GL_{n_i} (R)}(A_i)$.
\end{lemma}

We use \cref{lemma:blockcentral} to prove:

\begin{lemma}
\label{lemma:centralsplit}
Let $M$ be a stable matrix of $\mathfrak{g}_{s}$ such that $M=A+\pi B$ with $B \in Z(C_{\mathfrak{g}_{s}}(A))$ and $\bar{A}$ split matrix in Jordan canonical form then $C_{G_s}(M)=C_{G_s}(A)$.
\end{lemma}

\begin{proof}
First, we will assume that $A$ is a split primary matrix in Jordan canonical form. Since $M$ is stable, we can choose $A = \mu_s(a)I + \oplus_{i=1}^l N_{n_i}$ for some $a \in k$. The fact that $C_{G_s}(A) \subseteq C_{G_s}(M)$ follows from $B$ being an element of $Z(C_{\mathfrak{g}_{s}}(A))$. Let $g$ be an element of $C_{G_s}(M)$, we will show that $g \in C_{G_s}(A)$. Since $\bar{g}$ is an element of $C_{G_1}(\bar{A})$, we have that $g=g_1+\pi g_2$ with $g_1=\mu_s(\bar{g})$. Note $g_1$ is a block Toeplitz matrix that commutes with $A$ by \cref{lemma:singla}. Thus, we are left to show that $\pi g_2$ is an element of $C_{\mathfrak{g}_{s}}(A)$ which follows if $g_2$ is an element of $C_{\mathfrak{g}_{s-1}}(A)$. Note that $g_1$ commutes with $M$, as $C_{G_s}(A) \subseteq C_{G_s}(M)$. Also, since $g$ commute with $M$, we have that $g_2$ commute with $M$. Note that $M \mod{\mathfrak{p}^{s-1}}$ is still a stable matrix of $\mathfrak{g}_{s-1}$, as $B \mod{\mathfrak{p}^{s-1}}$ is still an element of $Z(C_{\mathfrak{g}_{s-1}}(A))$ by \cref{theorem:center}. By induction we can assume that $C_{G_{s-1}}(M)=C_{G_{s-1}}(A)$. Therefore, we have $g_2$ is an element of $C_{\mathfrak{g}_{s-1}}(A)$. Thus, we can conclude that $C_{G_s}(M)=C_{G_s}(A)$ when $A$ is a split primary matrix.

For the general case, let $A = \oplus_{i=1}^l A_i$ with $A_i$ a split primary matrix with distinct eigenvalues $\mu_s(a_i)$ with $a_i \in k$. We have that  $C_{G_s}(A) = \oplus_{i=1}^l
C_{GL_{n_i} (\mfO_s)}(A_i)$  by \cref{lemma:blockcentral}.
Similarly, $C_{G_s}(M) = \oplus_{i=1}^l
C_{GL_{n_i} (\mfO_s)}(M_i)$ with $M_i=A_i + \pi B_i$  a stable matrix. By the split primary case, we have that $C_{GL_{n_i} (\mfO_s)}(M_i)= C_{GL_{n_i} (\mfO_s)}(A_i)$, thus $C_{G_s}(M)=C_{G_s}(A)$.

\end{proof}

We take advantage of the previous two lemmas to prove the following: 

\begin{lemma}
\label{lemma:surjective}
Let $M=A+ \pi B$ a stable matrix of $\mathfrak{g}_s$ with $\bar{A}$ a split matrix in Jordan canonical form. Then for any $r \geq s$, the map $\rho_s: C_{G_r}(\mu_{s,r}(M)) \to C_{G_s}(M)$ is surjective.
\end{lemma}

\begin{proof}
Let $M=A+\pi B$ with $A$ a split matrix and $B \in Z(C_{\mathfrak{g}_s}(A))$. Note by \cref{lemma:centralsplit}, $C_{G_s}(M)=C_{G_s}(A)$. Thus, without loss of generality, we may assume $M=A$. Note  $A=\oplus_{i=1}^l A_{i}$ where each $A_i$ is a split primary matrix with distinct eigenvalues $\mu_s(a_i)$ with $a_i \in k$. By using $\mu_r(0)=0$ and $\mu_r(1)=1$, $\mu_{s,r}(A_i)$ is upper triangular with all diagonal entries equal to $\mu_r(a_i)$. Note $a_i-a_j$ is invertible in $k$ implies $\mu_r(a_i)-\mu_r(a_j)$ is invertible in $\mfO_r$. Using \cref{lemma:blockcentral}, we have $C_{G_r}(\mu_{s,r}(A)) = \oplus_{i=1}^l
C_{GL_{n_i} (\mfO_r)}(\mu_{s,r}(A_i))$ and similarly $C_{G_s}(A) = \oplus_{i=1}^l
C_{GL_{n_i} (\mfO_s)}(A_i)$. Thus to show surjectivity, it is enough to show surjectivity for each split primary matrix $A_i$. It follows by generalizing Singla's argument for $r=2$ in \cite[Lemma 5.1]{SinglaPooja2010Orog}.

\textbf{Split Primary Case:} Now we may assume that $A$ is a split primary matrix in Jordan canonical form. Since $M$ is stable, we can choose $A = \mu_s(a)I + \oplus_{i=1}^t N_{n_i}$ for some $a \in k$. Let $g \in C_{G_s}(A)$, by \cref{lemma:singla}, $g$ is an invertible block Toeplitz matrix of order $(n_1, n_2, \cdots , n_t)$ over the
ring $\mfO_s$. Our choice of section $\mu_{s,r}$ ensures that, $\mu_{s,r}(A) = \mu_r(a)I + \oplus_{i=1}^t N_{n_i}$, and $\mu_{s,r}(g)$ is an invertible block Toeplitz matrix of order $(n_1, n_2, \cdots , n_t)$ over the ring $\mfO_r$. Therefore, by \cref{lemma:singla} $\mu_{s,r}(g) \in C_{G_r}(\mu_{s,r}(A))$. Hence the map $\rho_s: C_{G_r}(\mu_{s,r}(A)) \to C_{G_s}(A)$ is surjective.
\end{proof}

\begin{proposition}
\label{proposition:stabilizer}
For $M=A+\pi B$ a stable matrix of $\mathfrak{g}_{l'}$ with $\bar{A}$ a split matrix in Jordan canonical form, $G_r(\psi_M)=K^{l'}C_{G_r}(\mu_{l',r}(M))$.
\end{proposition}

\begin{proof}
Note that $G_r(\psi_M)=\rho_{l'}^{-1}(C_{G_{l'}}(M))=K^{l'}C_{G_r}(\mu_{l',r}(M))$ where the second equality follows from the fact that the map $\rho_{l'}: C_{G_r}(\mu_{l',r}(M)) \to C_{G_{l'}}(M)$ is surjective, by \cref{lemma:surjective}. 
\end{proof}

\section{Even case}
\label{Section:5}

In this section, we will construct the stable irreducible representations of $G_r=\GL_{N}(\mathcal{O}_{r})$ when $r=2l$. In other words, we will construct the representations of $\Irr(G_r\mid \psi_M)$ where $M$ is a stable matrix with entries in $\mathcal{O}_{l}$. Our goal will be to prove that $\psi_M$ extends to a character $\tilde{\psi}_{M}$ of the stabilizer $G_r(\psi_M)$ by taking advantage of the following lemma. 



\begin{lemma}
\label{lemma:extension}
Let $G=N\cdot H$, a finite group where $N$ is a normal subgroup of $G$. Given $\chi$ a $G$-invariant irreducible one dimensional character of $N$ then the following are equivalent:
\begin{enumerate}
    \item $\chi$ extends to a character of $G$.
    \item $\chi$ restricted to a character of $[G,G] \cap N$ is the trivial character.
    \item $\chi$ restricted to a character of $[H,H] \cap N$ is the trivial character.
\end{enumerate}
\end{lemma}

\begin{proof}
(1) implies (2): Let $\hat{\chi}$ be the extension of $\chi$ to a character of $G$. $\hat{\chi}$ is also a $1$-dimensional character, and thus $\hat{\chi}$ is trivial on commutators. Therefore, we have that $\chi$ restricted to a character of $[G,G] \cap N$ is the trivial character.

(2) implies (1): Assume that $\chi$ is trivial on $[G,G] \cap N$. Thus we can think of $\chi$ as a character of the image of $N$ in the abelianization of $G$. Let $\hat{N}$ be this image. Since characters of abelian groups always extend, there is an extension of $\chi$, denoted as $\hat{\chi}$, to a character of $G/[G,G]$. Pre-composing $\hat{\chi}$ with the map $G\rightarrow G/[G,G]$, we obtain a one dimensional representation of $G$ that extends $\chi$.

Note (2) implies (3) follows as $[H,H] \cap N \subseteq [G,G] \cap N$. Thus, we are left to show that (3) implies (2). Let $g \in [G,G] \cap N$, we can write $g=\prod [x_i,y_i]$ with $x_i$ and $y_i$ in $G$. We can write any element $g$ of $G$ as $g=n_{g}\Bar{g}$ with $n_{g} \in N$ and $\Bar{g} \in H$. For any $x,y$ in $G$, we have $[x,y]=[n_{x}\Bar{x},n_{y}\Bar{y}]=A[\Bar{x},\Bar{y}]B$ where $A=n_{x}\Bar{x}n_{y}\Bar{x}^{-1}$ and $B=\Bar{y}n_{x}^{-1}\Bar{y}^{-1}n_{y}^{-1}$. Since $A$ and $B$ are in $N$ and $\chi$ is $G$-invariant, we have that $\chi(A)\chi(B)=1$.

 By the previous argument, we can write $g=\prod [x_i,y_i]=\prod A_i[\Bar{x_i},\Bar{y_i}]B_i$. We will let $g_i=[\Bar{x_i},\Bar{y_i}]$. Note $\chi(g)=\chi(A_1)\chi(g_1B_1g_1^{-1})\chi(g_1\prod  A_i[\Bar{x_i},\Bar{y_i}]B_i)=\chi(g_1\prod  A_i[\Bar{x_i},\Bar{y_i}]B_i)$, since $\chi(A_1)\chi(g_1B_1g_1^{-1})=1$ by $\chi$ being $G$-invariant. By induction, it is enough to assume we have $\chi({g'}A_ng_nB_n)=\chi(^{g'}A_n) \cdot \chi(^{{g'}g_n}B_n) \cdot \chi({g'}g_n)=\chi({g'}g_n)$. Note also by this process that the $\prod g_i$ is an element of $[H,H] \cap N$ and thus $\chi(g)=\chi(\prod g_i)=1$, concluding the proof. 

\end{proof}

\begin{remark}
We note that \cref{lemma:extension} can be obtained by a different argument. The equivalence of (1) and (2) is known from \cite[Proposition 7]{nagornyi1976complex}. The equivalence of part (3) follows from \cite[Lemma 5.4]{SinglaPooja2010Orog}. 
\end{remark}

Any element of the commutator of $G_r$ has determinant one. Thus, we will be interested in computing the determinant of elements of $K^{l'}$, where $l'=r-l$.

\begin{lemma}
\label{lemma:determinant}
Let $I+\pi^{l'}X$ be an element of $K^{l'}$ and $g \in G_r$ be such that $^gX=T+\pi Y$ with $T= (T_{i,j}) \in M_n(\mathcal{O}_{r})$ an upper triangular matrix. Then we have: 

\begin{equation*}
\det(I+\pi^{l'}X)=
    \begin{cases}
      1+\pi^{l'}\Tr(X) \text{ if $r$ is even}\\
      1+\pi^{l'}\Tr(X)+\pi^{2l'}\sum_{j=1}^n\sum_{i\neq j} T_{i,i}T_{j,j}/2 \text{ if $r$ and $p$ are odd}
    \end{cases}\,.
\end{equation*}

\end{lemma}
 
 \begin{proof}
 Note $|I+\pi^{l'}X|=|I+\pi^{l'}(T+\pi Y)|$ by conjugation with $g$. We denote:
 
\begin{equation}
|I+\pi^{l'}(T+\pi Y)| = 
\begin{vmatrix}
a_{1,1} & a_{1,2} & \cdots & a_{1,n} \\
a_{2,1} & a_{2,2} & \cdots & a_{2,n} \\
\vdots  & \vdots  & \ddots & \vdots  \\
a_{n,1} & a_{n,2} & \cdots & a_{n,n} 
\end{vmatrix}
\end{equation}
 
Doing a cofactor expansion of the first column we obtain the following:
$$|I+\pi^{l'}(T+\pi Y)| = \sum_{j=1}^{n} (-1)^{1+j} a_{j,1} M_{j,1},$$ where $M_{j,1}$ is a minor of the given matrix. We claim $a_{j,1} M_{j,1}=0$ for $j\neq 1$. When $j\neq 1$, $a_{j,1}=\pi^l Y_{j,1}$ since $T$ is an upper triangular matrix. Thus, it is enough to show that the minor $M_{j,1}$ is a multiple of $\pi^{l'}$ for $j\neq 1$. It follows by a cofactor expansion along the $j-1$ column of $M_{j,1}$ and noticing that all the entries of this column are a multiple of $\pi^{l'}$. Thus $|I+\pi^{l'}(T+\pi Y)| = a_{1,1} M_{1,1}$. Notice that the $M_{1,1}$ minor is just a smaller version of the matrix we started with. Thus by induction we can conclude that: $$|I+\pi^{l'}(T+\pi Y)| = \prod_{j=1}^n a_{i,i}=\prod_{j=1}^n (1+\pi^{l'}(T_{i,i}+\pi Y_{i,i}))=1+\pi^{l'}\Tr(X)+\pi^{2l'}\sum_{j=1}^n\sum_{i\neq j} T_{i,i}T_{j,j}/2.$$

Note that when $r$ is even $2l'=r$, so the formula for the even case follows.
\end{proof}

Recall that when $M=A + \pi B$ is a stable matrix with  $\bar{A}$ a split matrix in Jordan canonical form, we have  $G_r(\psi_M)=K^{l'}C_{G_r}(\mu_{l',r}(M))$ by \cref{proposition:stabilizer}. Moreover, $H:=C_{G_r}(\mu_{l',r}(M)) =C_{G_r}(\mu_{l',r}(A)) = \oplus_{i=1}^l
C_{GL_{n_i} (\mfO_r)}(\mu_{l',r}(A_i))$ is a direct sum of block matrices, by \cref{lemma:blockcentral}. We will denote $H_i=C_{GL_{n_i} (\mfO_r)}(\mu_{l',r}(A_i))$.

\begin{lemma}
\label{lemma:commutator}
Let $H$ be a subgroup of $G_r$ such that $H=\oplus_{i=1}^l H_{i}$, a direct sum of subgroups $H_{i}$ of $\mathrm{GL}_{n_{i}}(\mathcal{O}_{r})$. Given $I+\pi^lX= \oplus_{i=1}^l I+\pi^lX_i \in [H,H]$ with $I+\pi^lX_i \in H_i$ then $I+\pi^lX_i$ is an element of $[H_i,H_i]$. 
\end{lemma}

\begin{proof}
Let  $I+\pi^lX= \oplus_{i=1}^l I+\pi^lX_i = \prod_{j=1}^{r} [g_j,h_j]$ as $I+\pi^lX \in [H,H]$. Note that $[g_j,h_j]=[\oplus_{i=1}^lg_{i,j},\oplus_{i=1}^l h_{i,j}]=\oplus_{i=1}^l [g_{i,j},h_{i,j}]$  where $g_j=\oplus_{i=1}^l g_{i,j}$ and $h_j=\oplus_{i=1}^l h_{i,j}$. By fixing $i$ we have that  $I+\pi^lX_i= \prod_{j=1}^{r} [g_{i,j},h_{i,j}]$ and thus $I+\pi^lX_i \in [H_i,H_i]$.
\end{proof}

 We use the previous lemma to prove \cref{theorem:Ieven}, which we restate here for the reader's benefit. 

\begin{theorem A}
Let $r=2l$ and $\chi$ be a stable irreducible character of $G_r=\GL_{N}(\mathcal{O}_{r})$ above $\psi_M$, where $\psi_M$ is an irreducible character of $K^l$. Then:  

\begin{enumerate}
    \item $\psi_M$ extends to a character $\tilde{\psi}_{M}$ of the stabilizer $G_r(\psi_M)$.
    \item there is an irreducible character $\rho$ of $G_r(\psi_M)/K^l$ such that $\chi=\operatorname{Ind}_{G_r(\psi_M)}^{G_r}(\tilde{\psi}_{M}\rho)$.
\end{enumerate} 

\end{theorem A}

\begin{proof}

Since ${\psi}_{M}$ is a one dimensional character, to show that $\psi_M$ extends to a character $\tilde{\psi}_{M}$ of $G_r(\psi_M)=K^{l}C_{G_r}(\mu_r(A))$ it is enough to show that $\psi_M$ restricted to $[H,H] \cap N$ is trivial where $H=C_{G_r}(\mu_r(A))$ and $N=K^l$ by \cref{lemma:extension}. Since $\psi_{M}(I+\pi^{l}X) =\psi(\pi^{-{l}}\tr(MX)$, it is enough to show that $\tr(MX)=0 \mod{\mathfrak{p}^l}$ for $I+\pi^{l}X \in [H,H] \cap N$. 

We will prove such an extension exists by looking at three cases:

\textbf{Case 1: $A$ is a split primary matrix.}

Because we can pick $M$ up to conjugation, we may assume $M=A+\pi B$ is a stable matrix, with  $A=\lambda I \oplus   N_{n_i}$ where $N_{n_i}$ are principal nilpotent matrices and $B$ is an element of $Z(C_{\mathfrak{g}_l}(A))$. Let $I+\pi^lX$ be an element of  $[H,H] \cap N$, thus $\det(I+\pi^{l}X)=1$. By \cref{lemma:determinant}, as $l'=l$ we have $$\det(I+\pi^{l}X)=1+\pi^{l}\Tr(X),$$ so conclude that $$\Tr(X)=0 \mod{\mathfrak{p}^l}.$$ Since $X \in C_{g_l}(A)$, $X$ is a block upper-Toeplitz matrix (\cref{lemma:singla}).  Thus $$\Tr(MX)=\Tr(AX)+\pi \Tr(BX)=\lambda \Tr(X)+\pi \Tr(BX)=\pi \Tr(BX) \mod{\mathfrak{p}^l}.$$ Note since $B$ is an element of $Z(C_{\mathfrak{g}_l}(A))$, it has a unique value in the diagonal denoted as $b$. Thus by computation we have $$\Tr(MX)=\pi \Tr(BX)=\pi b\Tr(X)=0 \mod{\mathfrak{p}^l}.$$ Thus, the value of $\psi_M$ on $[H,H] \cap N$ is trivial. By \cref{lemma:extension} we conclude that $\psi_M$ extends to a character $\tilde{\psi}_{M}$ of $G_r(\psi_M)$. 

\textbf{Case 2: $A$ is a split matrix.}

By applying \cref{lemma:blockcentral}, we can assume that $M=A+\pi B= \oplus_{i=1} A_i+\pi B_i $ with each $A_i$ a split primary matrix and $B_i$ an element of $Z(C_{\mathfrak{g}_l}(A_i))$. Let $I+\pi^l X$ be an element of  $[H,H] \cap N$. By \cref{lemma:commutator}, $I+\pi^lX= \oplus_{i=1} I+\pi^l X_i$  with $I+\pi^l X_i \in [H_i,H_i]$ where $H= \oplus_{i=1} H_i$ with $H_i=C_{GL_{n_i} (\mfO_r)}(\mu_{l',r}(A_i))$. Note $$\Tr(MX)=\Tr(AX)+\pi \Tr(BX)=\sum_{i=1} \Tr(A_iX_i)+\pi \Tr(B_iX_i).$$ Since $A_i$ is a split primary matrix, $\Tr(A_iX_i)+\pi \Tr(B_iX_i)=0 \mod{\mathfrak{p}^l}$ by the argument in the case $1$. Thus, the value of $\psi_M$ on $[H,H] \cap N$ is trivial. We conclude that $\psi_M$ extends to a character $\tilde{\psi}_{M}$ of $G_r(\psi_M)$. 

\textbf{Case 3: $A$ is a non-split matrix.}

Since $M$ is a stable matrix by passing to an unramified extension $\mathcal{\hat{O}}_l$ of $\mathcal{O}_l$ we have that $M$ is conjugate to a matrix $A+\pi B$ with entries in $\mathcal{\hat{O}}_l$ such that $A$ is in Jordan canonical form with $B$ in $Z(C_{\hat{\mathfrak{g}_{l}}}(A))$. Let $\hat{\mathrm{F}}$ be the corresponding unramified extension of $\mathrm{F}$ and $\hat{\psi}:\hat{\mathrm{F}} \to \mathbb{C}^*$ be the corresponding character that extends $\psi:\mathrm{F} \to \mathbb{C}^*$ then as defined before we have $\hat{\psi}_M$ a character of $\hat{K_l}=\operatorname{Ker}(\GL_{N}(\mathcal{\hat{O}}_r) \to \GL_{N}(\mathcal{\hat{O}}_l))$. This can all be pictured in the following diagram:

\[
\begin{matrix}
  \hat{\mathrm{F}} &  \supset & \hat{\mfO}& \triangleright & \hat{\mathfrak{p}}=\pi\hat{\mfO} \\
       \vert & & \vert & & \vert \\
   \mathrm{F} &  \supset & \mfO & \triangleright & \mathfrak{p}=\pi \mfO
\end{matrix}
\qquad\qquad
\begin{matrix}
  \mathcal{\hat{O}}_r  \\ \vert  \\ \mathcal{O}_r
\end{matrix}
\]

Since $M$ splits over $\mathcal{\hat{O}}_l$, there is a character $\Hat{\chi}$ of $\hat{G_r}(\hat{\psi}_M)$ that extends $\hat{\psi}_M$. Define $\chi$ to be the restriction of $\Hat{\chi}$ to $G_r({\psi}_M)$ then $\chi$ is an extension of ${\psi}_M$.

Since there is an extension of ${\psi}_M$ to the stabilizer, then part (2) follows by an application of Clifford theory, \cref{thm:Clifford}.
\end{proof}

\section{Odd case}
\label{Section:6}

Throughout this section, we will assume $p > 2$ and $r=2l-1$. Thus $l'=l-1$. Recall we have the following chain of normal subgroups of $G_r$:

\[
G_{r}\supset K^{l'}\supset\ K^{l}=I+\mfp^{l}\mfg_{r}.
\]

Since $l \geq r/2 $, by \cref{lemma:Hillstabilizer} we have an explicit description of $\Irr(K^l)$. Given an irreducible character $\psi_M$  of $K^l$, we will determine $\Irr(K^{l'} \mid \psi_M)$ using a technique called Heisenberg lift. We will then use Clifford theory to describe $\Irr(G_r \mid \sigma)$, for certain irreducible characters $\sigma$ of $\Irr(K^{l'} \mid \psi_M)$.

\subsection{Heisenberg Lift}

Recall that the representation  $\psi_{M}:K^{l}\rightarrow\mathbb{C}^{\times}$ is defined by $\psi_{M}(I+\pi^{l}x)=\psi(\pi^{-{l'}}\tr(M x))$. From \cref{lemma:Kernel}, we also have an isomorphism $K^{l'}/K^{l}\cong \mfg_{1}$ by sending $I+\pi^{l'}x$ to $\bar{x}$, where the bars denote reductions mod $\mfp$. We can identify any subgroup of $K^{l'}$ which contains $K^{l}$ with a subgroup of $\mfg_{1}$. Define the alternating bilinear form:

\begin{align*}
B_{M}:K^{l'}/K^{l}\times K^{l'}/K^{l} & \longrightarrow\F_{q}\\
B_{M}((1+\pi^{l'}x)K^{l},(1+\pi^{l'}y)K^{l}) & =\tr(\bar{M}(\bar{x}\bar{y}-\bar{y}\bar{x})).
\end{align*}

Given this form, one can show that $\psi_M$ can be extended to a subgroup $R_{M}$ such that $K^{l}\subseteq R_{M}\subseteq K^{l'}$ and $R_{M}/K^{l}$ is the radical subspace for the form $B_{M}$ \cite[Subsection 3.2]{krakovski2018regular}. In fact, when $p > 2$, each such extension has an explicit formula in terms of a matrix $\hat{M} \in \mfg_{l}$ such that $M=\hat{M} \mod{\pi^{l'}}$  \cite[Lemma 6.2.]{stasinski2017representations}:

\begin{align*}
\psi_{\hat{M}}:R_{M} & \longrightarrow\mathbb{C}^{\times}\\
\psi_{\hat{M}}(1+\pi^{l'}x) & =\psi(\pi^{-{l}}\tr(\hat{M}(x-\frac{1}{2}\pi^{l'}x^2)).
\end{align*}

Each such extension $\psi_{\hat{M}}$ determines
uniquely an irreducible character of $K^{l'}$ that arises as follows.

\begin{proposition}
\label{proposition:krakovski}
\cite[Subsection 3.2]{krakovski2018regular} Let $J_{M}$ be a subgroup of $K^{l'}$ that contains $K^l$ such that $J_{M}/K^{l}$ is a maximal isotropic subspace for the form $B_{M}$. We have:

\begin{enumerate}
    \item  The function $\psi_{\hat{M}}$ defines a multiplicative character of $R_M$. Moreover, the same formula defines a character of $J_M$ that extends $\psi_{M} \in \Irr(K^l)$. 
    
    \item The character $\sigma = \Ind_{J_{M}}^{K^{l'}}(\psi_{\hat{M}})$ is irreducible and independent of $J_{M}$.
    \item  There is a bijection between extensions of $\psi_M$  to $R_M$ and $\Irr(K^{l'}\mid \psi_M)$. For each extensions $\psi_{\hat{M}}$ there is a unique irreducible character of $K^{l'}$ above $\psi_{\hat{M}}$, namely $\sigma = \Ind_{J_{M}}^{K^{l'}}(\psi_{\hat{M}})$.
\end{enumerate}
\end{proposition}

The previous assertions are well known and coined as 'Heisenberg lift,' see for example \cite[Proposition 8.3.3]{bushnell2006gauss}. However, we emphasize that the important fact is not just that $\psi_{M}$ has an extension to $J_M$ (this is true for any $p$, by \cite[Proposition 4.2]{HILL1995610}), but that there is an extension given by an explicit formula. In fact, this explicit formula is the essential reason for the assumption $p > 2$ in \Cref{theorem:Iodd} as we have divided by $2$.

\begin{definition}
\label{definition:super}
(Super Stable Character) We define an irreducible character $\chi$ of $G_r$ to be super stable if $\chi \in \operatorname{Irr}(G_r \mid \sigma)$ where $\sigma = \Ind_{J_{M}}^{K^{l'}}(\psi_{\hat{M}})$ with $\hat{M}$ a stable matrix of $\mathfrak{g}_{l}$.
\end{definition}

Our goal in this section is to give a construction of super stable irreducible characters of $G_r$. Thus, throughout this section let $\sigma = \Ind_{J_{M}}^{K^{l'}}(\psi_{\hat{M}})$ with $\hat{M}$ a stable matrix of $\mathfrak{g}_{l}$ and $M=\hat{M} \mod{\pi^{l'}}$. To apply Clifford's theory, we first need to understand the stabilizer of $\sigma$ in $G_r$. The proof of the following lemma follows the argument in \cite[Lemma 6.3]{stasinski2017representations}. 

\begin{lemma}
Let $\sigma = \Ind_{J_{M}}^{K^{l'}}(\psi_{\hat{M}})$ with $\hat{M}$ a stable matrix of $\mathfrak{g}_{l}$ with $\bar{A}$ a split matrix in Jordan canonical form. We have $G_r(\sigma)=G_r(\psi_{M})$ where $\psi_{M} \in \Irr(K^l)$ such that $M=\hat{M} \mod{\pi^{l'}}$.
\end{lemma}

\begin{proof}
If $g \in G_r(\sigma)$ then  $g$ fixes $\psi_{M}$ as the restriction of $\sigma$ to $K^l$ is a direct sum of copies of $\psi_{M}$. Thus, $G_r(\sigma) \subseteq G_r(\psi_{M})$. By \cref{proposition:stabilizer}, $$G_r(\psi_{M})=\rho_{l'}^{-1}(C_{G_{l'}}(M))=K^{l'}C_{G_{r}}(\mu_{l',r}(M))=K^{l'}C_{G_{r}}(\mu_{l,r}(\hat{M})).$$ Any element of $C_{G_{r}}(\mu_{l,r}(\hat{M}))$ fixes $\psi_{\hat{M}}$. Since $\sigma$ is the unique character in $\Irr(K^{l'}\mid \psi_{\hat{M}})$, $\sigma$ is stabilized by
$C_{G_{r}}(\mu_{l,r}(\hat{M}))$, and so $C_{G_{r}}(\mu_{l,r}(\hat{M}))K^{l'}\subseteq G_{r}(\sigma)$. Thus, $G_r(\sigma)=G_r(\psi_{M})$. 

\end{proof}

To understand the irreducible representations of $G_r(\sigma)$ above $\sigma$, we will show first that $\sigma$ has an extension to its stabilizer. To show such an extension exists, it is enough to show that $\sigma$ extends to $P$ a Sylow $p$-subgroup of $G_r(\sigma)$ by \cite[Lemma 4.8] {stasinski2017regular}. From now on, we will fix such an arbitrary subgroup $P$. By Lemma 3.4 in \cite{krakovski2018regular}, there exists a maximal isotropic subspace $J_{M}/K^l$ contained in $K^{l'}/K^l$ which is $P$-invariant. It follows that both $J_{M}$ and $\psi_{\hat{M}}$ are $P$-invariant under the conjugation action.  




\begin{lemma}
\label{lemma:J}
\cite[In Proof of Theorem 3.1]{krakovski2018regular}
Let ${\psi}_{\hat{M}}$ be a character of $J_{M}$ as in \cref{proposition:krakovski}. If ${\psi}_{\hat{M}}$ extends to $J_{M}P$ then $\sigma$ extends to $K^{l'}P$.
\end{lemma}
 


Thus, following \cref{lemma:J}, our goal will be to show that ${\psi}_{\hat{M}}$ extends to $J_{M}P$. 

\begin{proposition}
\label{proposition:M'extends}
Given ${\psi}_{\hat{M}}$ an irreducible character of $J_{M}$ with $\hat{M}$ a stable matrix with entries in $\mfO_l$ then ${\psi}_{\hat{M}}$ extends to $J_{M}P$.
\end{proposition}

\begin{proof}
Since ${\psi}_{\hat{M}}$ is a one dimensional character, we just need to show that ${\psi}_{\hat{M}}$ restricted to $[P,P] \cap J_{M}$ is the trivial character by \cref{lemma:extension}. Since $\psi_{\hat{M}}(I+\pi^{l'}X)=\psi(\pi^{-{l}}\tr(\hat{M}(X-\frac{1}{2}\pi^{l'}X^2))$, it is enough to have that $\tr(\hat{M}(X-\frac{1}{2}\pi^{l'}X^2))=0 \mod{\mathfrak{p}^l}$ for an arbitrary element $I+\pi^{l'}X \in [P,P] \cap J_{M}$. We will prove this proposition by looking at three cases.

\textbf{Case 1: $A$ is a split primary matrix.}

Since we can pick $\hat{M}$ up to conjugation, without loss of generality assume $\hat{M}=A+\pi B$ is a stable matrix, with  $A=\lambda I \oplus N_{n_i}$ where each $N_{n_i}$ is a principal nilpotent matrix over $\mfO_l$ and $B$ an element of $Z(C_{\mathfrak{g}_{l}}(A))$. Let $I+\pi^{l'}X$ be an element of  $[P,P] \cap J_{M}$, thus $\det(I+\pi^{l'}X)=1$ where $P$ is a $p$-Syllow subgroup of $C_{G_r}(\mu_l,r(\hat{M}))$. By \cref{lemma:determinant}, we have $$\det(I+\pi^{l'}X)=1+\pi^{l'}\Tr(X)+\pi^{2l'}\sum_{j=1}^n\sum_{i\neq j} T_{i,i}T_{j,j}/2$$ where $^gX=T+\pi Y$ with $T$ an upper triangular matrix. Thus, $$\Tr(X)+\pi^{l'}\sum_{j=1}^n\sum_{i\neq j} T_{i,i}T_{j,j}/2= 0 \mod{\pi^l}.$$ 
This implies $\Tr(T)=\Tr(X)=0 \mod{\pi}$ so $T_{j,j}=-\sum_{i\neq j} T_{i,i} \mod{\pi}$. Thus we obtain that: 
$$\sum_{j}^n\sum_{i\neq j} T_{i,i}T_{j,j}/2=-\sum_{j} T_{j,j}^2/2=-\Tr(T^2)/2=-\Tr(X^2)/2 \mod{\pi}.$$ 
So we conclude that,
$$\Tr(X)-\pi^{l'}\Tr(X^2)/2= 0 \mod{\pi^l},$$
as $$\pi^{l'}\sum_{j}^n\sum_{i\neq j} T_{i,i}T_{j,j}/2=-\pi^{l'}\Tr(X^2)/2 \mod{\pi^l}.$$

Since $X \in C_{\mathfrak{g} _l}(A)$, $X$ is a block upper-Toeplitz matrix by \cref{lemma:singla}. Working over $\mfO_l$ we have:  $$\Tr(\hat{M}(X-\frac{1}{2}\pi^{l'}X^2))=\Tr(A(X-\frac{1}{2}\pi^{l'}X^2)) + \pi \Tr(BX)= \lambda \Tr(X-\frac{1}{2}\pi^{l'}X^2) + \pi \Tr(BX)=\pi \Tr(BX).$$
Since $B$ is an element of $Z(C_{\mathfrak{g}_{l}}(A))$, it has a unique value in the diagonal denoted as $b$ by \cref{theorem:center}. Thus by computation $\pi \Tr(BX)=\pi b\Tr(X)=0$ as $\Tr(X)=0 \mod{\pi^{l'}}$. Thus, the value of $\psi_{\hat{M}}$ on $[P,P] \cap N$ is trivial and by \cref{lemma:extension} we conclude that $\psi_{\hat{M}}$ extends to a character of $J_{M}P$. 

\textbf{Case 2: $A$ is a split matrix.}

By applying \cref{lemma:blockcentral}, we can assume that $\hat{M}=A+\pi B= \oplus A_i+\pi B_i $ with each $A_i$ a split primary matrix and $B_i$ an element of $Z(C_{\mathfrak{g}_{l}}(A_i))$. Let $I+\pi^{l'}X$ be an element of  $[P,P] \cap J_{M}$. By \cref{lemma:commutator},  $I+\pi^{l'}X= \oplus_{i=1}^l I+\pi^{l'}X_i$  with $I+\pi^{l'}X_i \in [P_i,P_i]$ where $P= \oplus P_i$ with $P_i\subset C_{G_r}(A_i)$. By letting $\hat{M}_i=A_i+\pi B_i$ we have that $$\Tr(\hat{M}(X-\frac{1}{2}\pi^{l'}X^2))=\sum \Tr(\hat{M}_i(X_i-\frac{1}{2}\pi^{l'}X_i^2)).$$ Since $A_i$ is a split primary matrix, $\Tr(\hat{M}_i(X_i-\frac{1}{2}\pi^{l'}X_i^2))=0$ by the argument in case 1. Thus, the value of $\psi_{\hat{M}}$ on $[P,P] \cap J_{M}$ is trivial. We conclude that $\psi_{\hat{M}}$ extends to a character of $J_{M}P$. 

\textbf{Case 3: $A$ is a non-split matrix.}

Since ${\hat{M}}$ is a stable matrix by passing to an unramified extension $\mathcal{\hat{O}}_l$ of $\mathcal{O}_l$ we have that ${\hat{M}}$ is conjugate to a matrix $A+\pi B$ with entries in $\mathcal{\hat{O}}_l$ such that $A$ is in Jordan canonical form with $B$ an element of $Z(C_{\hat{\mathfrak{g}_{l}}}(A))$. Let $\hat{\mathrm{F}}$ be the corresponding unramified extension of $F$ and $\hat{\psi}:\hat{\mathrm{F}} \to \mathbb{C}^*$ be the corresponding character that extends $\psi:\mathrm{F} \to \mathbb{C}^*$ then as defined before we have $\hat{\psi}_{\hat{M}}$ a character of $\hat{J}_M$. Since ${\hat{M}}$ splits over $\mathcal{\hat{O}}_l$, there is a character $\Hat{\chi}$ of $\hat{J}_M\hat{P}$ that extends $\hat{\psi}_{\hat{M}}$. Define $\chi$ to be the restriction of $\Hat{\chi}$ to $J_{M}P$ then $\chi$ is an extension of ${\psi}_{\hat{M}}$.

\end{proof}

 We use the previous proposition to finish the proof of \cref{theorem:Iodd}, which we restate here for the reader's benefit. 
 
\begin{theorem B}
Assume that $\mathcal{O}$ has residue field of characteristic $p>2$ and $r=2l-1$. Let $\chi$ be a super stable irreducible character of $G_r$, above an irreducible character $\sigma$ of $K^{l'}$. Then:
\begin{enumerate}
    \item $\sigma$ extends to a character $\tilde{\sigma}$ of the stabilizer $G_r(\sigma)$.
    \item there is an irreducible character $\rho$ of $G_r(\sigma)/K^{l'}$ such that $\chi=\operatorname{Ind}_{G_r(\sigma)}^{G_r}(\tilde{\sigma}\rho)$.
\end{enumerate} 

\end{theorem B}

\begin{proof}
By \cref{proposition:M'extends} and \cref{lemma:J}, we have that $\sigma$ extends to a character $\tilde{\sigma}$ of the stabilizer $G_r(\sigma)$. Since there is an extension of $\sigma$ to the stabilizer, part (2) follows by an application of Clifford theory, \cref{thm:Clifford}. 
\end{proof}

\section{Remarks on a result of Hill}
\label{Section:7}

In this section, we will assume that $r= 2l-1$ is odd and that $M$ is a strongly semisimple matrix over $\mathcal{O}_{l'}$. Hill's Proposition 3.6 of \cite{hill1995semisimple} tries to classify all irreducible representations of $G_r$ above $\psi_M \in \Irr(K^l)$, which we now state. 

\begin{statement}
\cite[Proposition 3.6]{hill1995semisimple}
\label{proposition:Hillodd}
Let $r \geq 3$ be odd, let $\chi$ be a strongly semisimple irreducible character of $G_r$ above $\psi_M \in \Irr(K^l)$ with $M$ a strongly semisimple matrix over $\mathcal{O}_{l'}$. Then there exist an irreducible $\sigma$ of $\Ind_{K^l}^{K^{l'}}(\psi_M)$ and an extension of $\sigma$ to $G_r(\psi_M)$ such that $\chi=\Ind_{G_r(\psi_M)}^{G_r}(\sigma \rho)$, where $\rho$ is an irreducible character  of $G_r(\psi_M)/K^{l'}$.
\end{statement}

Unfortunately, the proof in \cref{proposition:Hillodd} contains an error. We have a counterexample to \cref{proposition:Hillodd}. For this counterexample, we will fix $M=aI$, a scalar matrix over $\mathcal{O}_{l'}$. Thus any irreducible character $\chi$ of $G_r$ above $\psi_M$ is strongly semisimple. Since $M$ is a scalar matrix, we have $G_r(\psi_M)=G_r$, and the alternating bilinear form $B_M$ gives us that the subgroup $R_{M}$ is equal to $K^{l'}$. Thus by "Heisenberg Lift" $\psi_M$ has extensions to $K^{l'}$. When $p > 2$, recall each such extension has an explicit formula in terms of a matrix $\hat{M} \in \mfg_{l}$ such that $M=\hat{M} \mod{\pi^{l'}}$:

\begin{align*}
\psi_{\hat{M}}:K^{l'} & \longrightarrow\mathbb{C}^{\times}\\
\psi_{\hat{M}}(1+\pi^{l'}x) & =\psi(\pi^{-{l}}\tr(\hat{M}(x-\frac{1}{2}\pi^{l'}x^2)).
\end{align*}

In this case any irreducible $\sigma$ of $\Ind_{K^l}^{K^{l'}}(\psi_M)$ is exactly equal to $\psi_{\hat{M}}$ for some $\hat{M}$ as above. 

\begin{proposition}
Let $M=aI$ be a scalar matrix over $\mathcal{O}_{l'}$ with $\sigma=\psi_{\hat{M}}$ an irreducible character of $K^{l'}$ above $\psi_M$. We have $G_r(\sigma)=\rho_{l}^{-1}(C_{G_{l}}(\hat{M}))$ for any $\hat{M} \in \mfg_{l}$ such that $M=\hat{M} \mod{\pi^{l'}}$. 
\end{proposition}

\begin{proof}
Note that the following map is a bijection: 

\begin{align*}
\mfg_{l} & \longrightarrow \mfg_{l}\\
x & \longmapsto x-\frac{1}{2}\pi^{l'}x^2.
\end{align*}
Thus if $g \in G_r(\sigma)$ then by a change of variable from  $x-\frac{1}{2}\pi^{l'}x^2$ to $y$ we have:
$$\tr(\hat{M}y)=\tr(g\hat{M}g^{-1}y) \mod \pi^l$$ for any $y \in \mfg_{l}$. Since the trace form is non-degenerate, we have $g \in \rho_{l}^{-1}(C_{G_{l}}(\hat{M}))$.
\end{proof}

We will now fix $\hat{M}=M+\pi^{l'} B$ such that $B$ is not an element of  $Z(C_{\M_{N}(\mathcal{{O}}_{l'})}(aI))$, the center of $ \mathfrak{g}_{l}$. Thus $B$ is not a scalar matrix. Note that with this choice of $B$ we are not in the super stable case. Under such conditions, we have $G_r(\sigma)$ is a proper subset of $G_r(\psi_M)=G_r$, as there exists an element $g$ that does not commute with $\hat{M}$. Thus, it is not possible for $\sigma$ to have an extension to $G_r(\psi_M)=G_r$, as claimed in \cref{proposition:Hillodd}. This gives a counterexample to the construction of \cref{proposition:Hillodd}. 

To better understand where the argument fails in \cref{proposition:Hillodd}, we take a moment to recall Hill's approach. Hill's construction is based on a counting argument to show that all the irreducible constituents of $\Ind_{K^l}^{K^{l'}}(\psi_M)$ extend to $G_r(\psi_M)$. The proof fails where it is claimed that the number of minimal degree characters of $G_r$ above $\psi_{\bar{M}} \in \Irr(K^{r-1})$ is bounded above by the number of choices for $\psi_{M} \in \Irr(K^{l})$ times the number of extensions of an irreducible constituent of $\Ind_{K^l}^{K^{l'}}(\psi_M)$ to $G_r(\psi_M)$. We note it is not clear why only one irreducible constituent of $\Ind_{K^l}^{K^{l'}}(\psi_M)$ is needed for this upper bound, as there could be several irreducible constituents. Unless the irreducible constituents are all conjugate under $G_r(\psi_M)$, each one will have distinct irreducible representations lying above them. Note by comparing the counterexample we gave with a $\sigma$ coming from the super stable case, we obtain examples of different irreducible constituents of $\Ind_{K^l}^{K^{l'}}(\psi_M)$ that are not conjugate under $G_r(\psi_M)$. 

One of the goals of Hill's paper \cite{hill1995semisimple} is to construct the characters of minimal degree in the semisimple case. We point out that the construction of the super stable gives constructions of characters of minimal degree of $G_r$ above $\psi_A$ a character of $K^{r-1}$. An interesting question is whether there are characters of minimal degree which are not stable and, if so, how to construct them.

\vspace{2mm}

\bibliographystyle{plain}
\bibliography{THM_1.bib}

\begin{thebibliography}{10}

\bibitem{bushnell2006gauss}
C.J. Bushnell and A.~Fr{\"o}hlich.
\newblock {\em Gauss Sums and $p$-adic Division Algebras}.
\newblock Lecture Notes in Mathematics. Springer Berlin Heidelberg, 2006.

\bibitem{bushnell1993admissible}
C.J. Bushnell and P.~C. Kutzko.
\newblock The admissible dual of $\mathrm{GL}(\text{N})$ via compact open
  subgroups.
\newblock {\em Annals of Mathematics Studies}, (129), 1993.

\bibitem{HILL1995610}
Gregory Hill.
\newblock Regular elements and regular characters of
  $\mathrm{GL}_n(\mathcal{O})$.
\newblock {\em Journal of Algebra}, 174(2):610 -- 635, 1995.

\bibitem{hill1995semisimple}
Gregory Hill.
\newblock Semisimple and cuspidal characters of $\mathrm{GL}_n(\mathcal{O})$.
\newblock {\em Communications in Algebra}, 23(1):7--25, 1995.

\bibitem{isaacs2006character}
I.M. Isaacs.
\newblock {\em Character Theory of Finite Groups}.
\newblock AMS Chelsea Publishing Series. AMS Chelsea Pub., 2006.

\bibitem{krakovski2018regular}
Roi Krakovski, Uri Onn, and Pooja Singla.
\newblock Regular characters of groups of type $a_n$ over discrete valuation
  rings.
\newblock {\em Journal of Algebra}, 496:116--137, 2018.

\bibitem{nagornyi1976complex}
S.~V. Nagornyj.
\newblock Complex representations of the group $\mathrm{GL}(2,\mathbb{Z} /p^n
  \mathbb{Z})$.
\newblock {\em Zapiski Nauchnykh Seminarov POMI}, 64:95--103, 1976.

\bibitem{SinglaPooja2010Orog}
Pooja Singla.
\newblock On representations of general linear groups over principal ideal
  local rings of length two.
\newblock {\em Journal of Algebra}, 324(9):2543--2563, 2010.

\bibitem{stasinski2017representations}
Alexander Stasinski.
\newblock Representations of $\mathrm{GL}_n$ over finite local principal ideal
  rings: an overview.
\newblock American Mathematical Society, 2017.

\bibitem{stasinski2017regular}
Alexander Stasinski and Shaun Stevens.
\newblock The regular representations of $\mathrm{GL}_n$ over finite local
  principal ideal rings.
\newblock {\em Bulletin of the London Mathematical Society}, 49(6):1066--1084,
  2017.

\end{thebibliography}

\end{document}